\documentclass[10pt]{article}
\usepackage{graphicx,pstricks}
\usepackage{amsmath,amsfonts,amssymb,amsthm}

\theoremstyle{plain}

\newtheorem*{prop}{Proposition}
\newtheorem*{schur}{Axel Schur (1921)}
\newtheorem{lem}{Lemma}

\theoremstyle{definition}

\newtheorem*{remark*}{Remark}
\newcommand{\bC}{\mathbf C}

\newcommand{\R}{{\mathbb R}}

\newcommand{\RR}{\mathbb R}

\newcommand{\ZZ}{\mathbb Z}

\newcommand{\wh}{\widehat}

\newcommand{\vt}{\vartheta}

\newcommand{\ed}{\end{document}}

\title{Self-Intersecting Periodic Curves in the Plane}
\author{J. Howie \& J. F. Toland}
\date{}

\begin{document}
\maketitle
\newcommand{\wt}{\widetilde}
\begin{abstract} \noindent Suppose a smooth planar curve $\gamma$   is $2\pi$-periodic in the $x$ direction and the length of one period is $\ell$. It is shown that if $\gamma$ self-intersects, then it has  a segment of length $\ell- 2\pi$ on which it self-intersects and  somewhere its curvature is at least $2\pi/(\ell - 2\pi)$. The proof involves the projection $\Gamma$ of $\gamma$ onto a cylinder. (The complex relation between $\gamma$ and $\Gamma$ was recently observed analytically in \cite{apos}, see also \cite[Ch. 10]{stein}). When $\gamma$ is in general position there is a bijection between self-intersection points of $\gamma$ modulo the periodicity, and  self-intersection points of $\Gamma$ with winding number 0 around the cylinder. However, our proof depends on the observation that a loop in $\Gamma$ with winding number 1 leads to a self-intersection point of $\gamma$.

\end{abstract}

Let  a smooth $2\pi$-periodic  curve $\gamma$ in the $(x,y)$-plane be parametrized by arc-length as follows:
\begin{gather*}\label{curve}\begin{cases}
\gamma= \{ p(s): s \in \R\},\quad  p(s)=(u(s),v(s)),
\\u(s+\ell) = 2\pi+u(s),\quad v(s+\ell) = v(s),
\\{u'(s)}^2 + {v'(s)}^2 =1,\end{cases}\quad s \in \R.\end{gather*}
The length of one period of $\gamma$ is $\ell$ and $q \in \gamma$ is called a {\em crossing} if $q=p(s_1)=p(s_2)$ and $s_1 \neq s_2$. Note that crossings exist if and only if $p$ is not injective.
   A crossing  $q$ is called {\em simple} if   there are exactly two real numbers $s_1\neq s_2$ with $p(s_1)=p(s_2)=q$ and
if  $p'(s_1) \neq p'(s_2)$ when $p(s_1)=p(s_2)$ and $s_1\neq s_2$.
%
Note that the smooth curve $\gamma$ can be approximated arbitrarily closely by smooth curves
in {\em general position}, that is with all crossings simple.
If $\gamma$ is in general position,
then it follows from the smoothness that the set of crossings is discrete, and hence finite by compactness.
Let $p'(s) = (\cos \vt(s),\sin \vt(s)),~s \in \RR$, where $\vt$ is smooth \cite[Prop. 2.2.1]{pres}.    The  goal is to establish the following which is intuitively obvious. (A {\em periodic segment} of $\gamma$ is a segment of the form $\{p(t): t \in [a,a+\ell]\}$.)
\begin{prop}  Suppose  that all crossings of $\gamma$ are simple.

(a) If $p$ is injective on every interval of length $\ell-2\pi$, $p$ is injective.

(b) If $p$ is not injective  its curvature is somewhere at least $2\pi/(\ell -2\pi)$.

(c) If $p$ is not injective  and $\vt$ is periodic, then  $\gamma$ has a periodic segment which contains two  crossings.
\end{prop}

The global problem of bounding from below the maximum curvature of a self-intersecting periodic planar curve arose in a study of water waves beneath an elastic sheet. 
In the model \cite{To7}, the sheet energy increases with the curvature and, roughly speaking, the conclusion needed was that sheets of certain energies could not self-intersect.

\begin{remark*}
Periodicity of $\vt$ in the Proposition
does not follow from that of $p$, as the first diagram below shows. Part (c)
of the Proposition is illustrated in the second diagram, where $\vt$ is periodic.

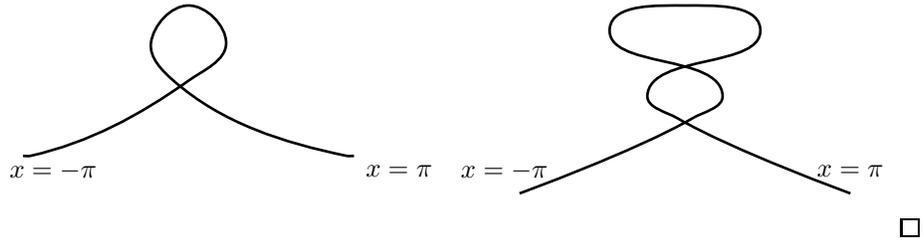
\begin{figure}[h]
\pspicture(4,.7)(16,3.6)
\pscurve[linewidth=1pt](4.2, 1)(4.3,1)(6.4,2)(6.9,2.5)(6.4,3)(5.9,2.5)(6.2,2)(8.5,1)(8.6,1)     %
\rput(4.6,.8){$x= -\pi$}%
\rput(9.2,.8){$ x=\pi$}%
\pscurve[linewidth=1pt]
(10.8,.5)(13.1,1.5)(13.5,1.8)(12,2.7)
(13,3)
(14,2.7)(12.5,1.8)(12.9,1.5)(15.2,.5)
\rput(10.6,.8){$x= -\pi$}%
\rput(15.2,.8){$x= \pi$}%
\endpspicture
\end{figure}

\qed
\end{remark*}


 For a proof, we project $\gamma$ onto the cylinder $\bC= S^1\times \RR$, where $S^1= \{ e^{i\phi}: \phi \in \RR\}$. Let  $P: \RR \to \ \bC$  be given by $P(s) = (e^{iu(s)},v(s))$ and let
$\Gamma = \{P(s): s \in [0,\ell]\}.$  Thus the projection of  the periodic, non-compact curve $\gamma$  in $\RR^2$ onto $\bC$ is the compact curve $\Gamma$.  Now $\Gamma$ has a crossing $Q$ if
 $P(s_0) = P(t_0)=Q$ for some $0\leq t_0<s_0 < \ell$ and we note that
  \begin{gather*}P(s_0) = P(t_0)
  \text{ if and only if }
p(s_0) = p(t_0) +k(2\pi ,0) = p(t_0+k\ell), \quad k \in \ZZ,
\end{gather*}  where  $k=\#(\Gamma_Q)$, the winding number  around $\bC$  of \begin{equation}\label{gammaQ}\Gamma_Q = \{ P(s): s \in [t_0,s_0]\},\end{equation} a loop at $Q$.
Crossings of $\Gamma$ with winding number $k$ correspond to the existence of horizontal chords  with length $2|k|\pi$ connecting points of $\gamma$.  Significantly for the Proposition, there is a one-to-one correspondence between crossings of $\gamma$ and crossings of $\Gamma$ with winding number zero. Note that  $\#(\Gamma) =1$, since $P(\ell) = P(0)$ and $p(\ell) = p(0)+(2\pi,0)$.

\begin{lem}\label{l2}
Suppose that $\#(\Gamma_Q) \in \{0,1\}$ for a crossing $Q$ of $\Gamma$. Then $p$ is not injective on some  interval of length $\ell$.
\end{lem}\begin{proof}
   By hypothesis
 $\Gamma_Q:=\{ P(s):s\in   [t_0,s_0]\}$, $[t_0,s_0]\subset [0,\ell)$  and
$$u(s_0) = u(t_0) + 2k\pi \text{ for }k \in \{0,\, 1\},\quad v(s_0)=v(t_0).
$$
If $k=0$, $p(s_0)=p(t_0)$ and  the conclusion holds. If $k=1$,  $$p(s_0) = p(t_0+\ell),\quad 0< t_0+\ell-s_0 < \ell,$$
and again the conclusion holds.
\end{proof}

%
\begin{remark*} Note that if $\#(\Gamma_Q)=-1$, the proof of   Lemma \ref{l2}
leads only to the conclusion that there is an interval  of length $2\ell$
on which $p$ is not injective, as illustrated in the example below.


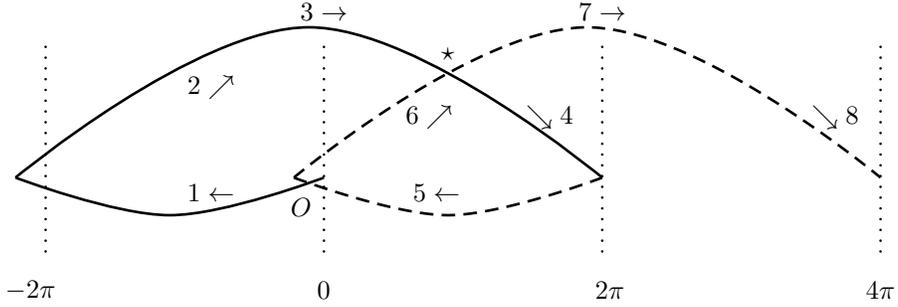
\pspicture*(4.8,-2.7)(17,1.7)
\pscurve[linewidth=1pt](4.9, -1)(6.95,-1.5)(9,-1)%
\pscurve[linewidth=1pt,linestyle=dashed](8.6,-1)(10.65,-1.5)(12.7,-1)
\pscurve[linewidth=1pt](4.9,-1)(8.8,1)(12.7,-1)
\pscurve[linewidth=1pt,linestyle=dashed](8.6,-1)(12.5,1)(16.4,-1)     %
\psline[linewidth=1pt,linestyle=dotted,linecolor=black](12.7,-2)(12.7,.75)%
\psline[linewidth=1pt,linestyle=dotted,linecolor=black](5.3,-2)(5.3,.75)%
\rput(5.1,-2.5){$ -2\pi$}%
\rput(12.8,-2.5){$ 2\pi$}%
\rput(7.5, -1.2){$1\leftarrow$}
\rput(7.5, .2){$2\nearrow$}
\rput(12, -.2){$\searrow 4$}
\rput(15.8, -.2){$\searrow 8$}
\rput(10.5, -1.2){$5\leftarrow$}
\rput(10.4, -0.2){$6\nearrow$}
\rput(9, 1.2){$3\to$}
\rput(12.7, 1.2){$7\to$}
\rput(10.65, .65){$\star$}
\rput(8.7, -1.4){$O$}
\psline[linewidth=1pt,linestyle=dotted,linecolor=black](16.4,-2)(16.4,.75)%
\rput(16.4,-2.5){$ 4\pi$}%
\psline[linewidth=1pt,linestyle=dotted,linecolor=black](9,-2)(9,.75)%
\rput(9,-2.5){$ 0$}%
\endpspicture

%
%
The segment $1\to 2\to 3\to 4$, in which arrows denote increasing arc-length, represents one period of $\gamma$ in $\RR^2$. The dashed curve $5\to 6\to 7\to 8$ represents the next period. The segment numbered $1$ contains a sub-loop of $\Gamma$  on $\bC$ with winding number $-1$
  and the construction just described leads to the
crossing $O$ on $\gamma$. However, the length of the corresponding closed sub-arc of
$1\to 2\to 3\to 4 \to 5$ in $\R^2$  lies between $\ell$ and $2\ell$ which does not vindicate the Proposition. However, there is another crossing
  $\star$ on $\gamma$, and the closed loop $4\to5\to 6$  satisfies the conclusion of the Proposition.
\qed
\end{remark*}

 The following is the key.

\begin{lem}\label{key}Suppose the crossings of $\Gamma$ are all simple. For any loop at $\wt Q$ of the form   $\Gamma_{\wt Q} = \{P(s): s\in [a,b] \}$, $P(a)=P(b)=\wt Q$,   with
  $\#(\wt \Gamma)>1$,  there exists  a sub-loop  at $\wt Q_1$  of the form   $\Gamma_{\wt Q}:= \{P(s): s \in [a_1,b_1]\}$, $P(a_1)=P(b_1)=\wt Q_1$,  $a\leq a_1<b_1 < b$, with $\#(\Gamma_{\wt Q_1}) =1$.\end{lem}

\begin{proof} Since $\#(\Gamma_{\wt Q})>1$ it follows from the topology of the cylinder that $\Gamma_{\wt Q}$ has a crossing.
The proof is by induction on the number of crossings.

If $\Gamma_{\wt Q}$ has only  one crossing,
 $\Gamma_{\wt Q}$ is the union of two loops,  $\wt\Gamma_1$ and $\wt\Gamma_2$,  based at a point of $ \Gamma_{\wt Q}$. Since they have  no crossings, each  has winding number  $\pm 1$ or $0$. Since the sum of their winding numbers  is $\#(\Gamma_{\wt Q}) > 1$,  each has winding number 1 and $\#(\Gamma_{\wt Q})=2$. If   $ \wt Q \in \widetilde\Gamma_2$, then the sub-path $\widetilde\Gamma_1$ satisfies the conclusion of the lemma, and vice versa.

Now we make the inductive hypothesis that the lemma  holds for any  loop  $\Gamma_{\wt Q}$ of the form in the lemma with no more than $N-1$ crossings, $N \geq 2$.

Suppose a loop  $\Gamma_{\wh Q}=\{P(s): s \in [\wh a,\wh b]\}$, $P(\wh a)=P(\wh b)= \wh Q$,  has $N$ crossings. Choose one of them,  $P(s_1) = P(t_1)=:\wt Q$, say. This splits  $\Gamma_{\wh Q}$ into two loops, $\wt\Gamma_1$ and $\wt\Gamma_2$, based at $\wt Q$. If they both have winding number 1, then the result follows, exactly as in the case $N=1$ above.    Otherwise one of them, $\wt\Gamma_1$ say, has winding number at least  2 and no more than  $N-1$ crossings.

Now, momentarily, let $\wt Q$ be the origin of arc length so that $\wt\Gamma_{1} = \{P(s): s \in [0,\wt t\,]\}$ where $s$ is arc length measured from $\wt Q$  along $\wt \Gamma_1$.  Then, by induction, there is a loop $\wt\Gamma_{11}$ in $\wt\Gamma_1$, satisfying  the conclusion of the lemma with $[0,\wt t\,]$ instead of $[a,b]$, and winding number 1.

If $\wt\Gamma_{11}$ does not contain $\wh Q$, then $\wt\Gamma_{11}$ with the original parametrization satisfies the conclusion of the lemma.

   If $\wt\Gamma_{11}$ does contain $\wh Q$, then its complement in $\wh\Gamma$ is a sub-path $\wt\Gamma_{12}= \{P(s): s \in [a',b']\subset [a,b]\}$ of $\wh \Gamma$, with winding  number  not smaller than 1 and  no more than $N-1$ crossings.

If the winding number of $\wt\Gamma_{12}$ is 1, then we are done. If it exceeds 1, then the required conclusion follows from the inductive hypothesis.
\end{proof}

\begin{lem}\label{howie} If
 $\#(\Gamma_Q)>1 $ for a crossing $Q$ of $\Gamma$, then $p$ is not injective on some closed interval of length $\ell$.
\end{lem}
\begin{proof}
Assume first that all the crossings of the original curve $\Gamma$ are simple.   Putting $\wt \Gamma=\Gamma_Q$ in  Lemma \ref{key}  gives the existence of a crossing of $\Gamma$ with winding number 1. The required result follows by  Lemma \ref{l2} when all the crossings of $\Gamma$ are simple.
If the crossings of $\Gamma$ are not all simple, apply the conclusion of Lemma \ref{key}
%
%
to a uniform periodic approximation $\gamma_1$  of $\gamma$  parametrized by a
smooth
periodic function $p_1$  with the property that each crossing of $\Gamma_1$  is simple
%
and close to a crossing of $\Gamma$.
The required result in the general case will follow by a simple limiting argument.
\end{proof}

\emph{Proof of the Proposition}. (a)
If $p$ is not injective, $\Gamma$ has a crossing, $Q$. Suppose $P(t_0)=P(s_0)$, $0\leq t_0<s_0<\ell$, Then, in the  notation of \eqref{gammaQ},
$\Gamma_Q = \{ P(s): s \in [t_0,s_0]\}$ and there is a minimal sub loop $\Gamma_{Q_1}=\{P(s): s \in [t_1,s_1]\}$ of $\Gamma_Q$ (a loop in $\Gamma_Q$ which has no proper sub loop)  $[t_1,s_1]\subset [t_0,s_0]$, $P(s_1)=P(t_1)=:Q_1$. Since
 $\Gamma_{Q_1}$ has no crossings,  $|\#(\Gamma_{Q_1}) |\leq 1$.

Now  we observe that if $p$ is not injective, then it is not injective on some interval of length $\ell$.
 If
 $\#(\Gamma_{Q_1}) \in \{0,1\}$, the observation holds by Lemma \ref{l2}. If $\#(\Gamma_{Q_1}) = -1$, since $\#(\Gamma) =1 $,  the complement of $\Gamma_{Q_1}$ in $\Gamma$ has winding number 2 and the observation holds,  by Lemma \ref{howie}.

 Now consider an interval  $[a,a+\ell]$ on which $p$ is not injective. Since $p(a+\ell) = p(a) + (2\pi,0)$, it follows easily (from the diagram below!) that the length of any loop in this periodic segment of $\gamma$ does not exceed $\ell-2\pi$. Hence there is an interval of length $\ell-2\pi$ on which $p$ is not injective.

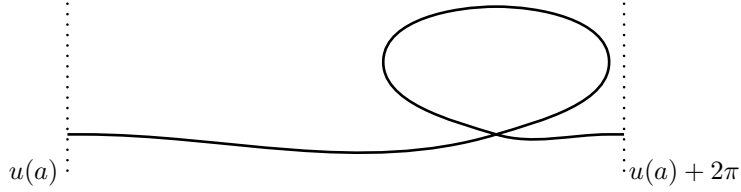
\pspicture*(4.8,-1.5)(17,1.7)
\pscurve[linewidth=1pt](6.8,-.5)(7,-.5)(12.5,-.5)(14,.5)(12.5,1.2)(11,.5)(12.5,-.5)(14,-.5)(14.2,-.5)     %
\psline[linewidth=1pt,linestyle=dotted,linecolor=black](14.2,-1)(14.2,1.25)%
\psline[linewidth=1pt,linestyle=dotted,linecolor=black](6.8,-1)(6.8,1.25)%
\rput(6.35,-1){$u(a)$}
\rput(15,-1){$u(a) + 2\pi$}
\endpspicture

 (b) A  classical result \cite{aschur} in the case of plane curves is the following  \cite[Remark on p. 38]{chern}.
\begin{schur}Suppose that $\Upsilon_i=\{\upsilon_i(s):s \in [0,S]\}$, $i=1,\,2$, are two plane curves parametrized by arc length, with  the same length $S$ and with   curvatures $\kappa_i(s)$ at $\upsilon_i(s)$. Suppose that $\Upsilon_1$ has no self-intersections and, along with the chord from $\upsilon_1(0)$ to $\upsilon_1(S)$, bounds a convex region. Furthermore, suppose that $|\kappa_2| \leq \kappa_1$ on $[0,S]$. Then $|\upsilon_2(s) - \upsilon_2(0)| \geq |\upsilon_1(s) - \upsilon_1(0)|$, $s \in [0,S]$.

\end{schur}
Let $\Upsilon_2$ be a closed loop in $\gamma$ with length $S$ no greater than $\ell-2\pi$ and suppose that at every point its curvature $|\kappa_2| \leq 2\pi(1-\epsilon)/(\ell-2\pi)$ for some $\epsilon >0$.  Let $\Upsilon_1$ be the segment of length $S$ of  a circle of radius  $(\ell-2\pi)/(2\pi(1-\epsilon))$. Now $|\kappa_2| \leq \kappa_1$,  $\Upsilon_1$ is not closed but $\Upsilon_2$ is closed, which contradicts  Schur's result. Hence no such $\epsilon$ exists, which  proves (b).

 (c) Consider a periodic segment of $\gamma$ with only one crossing at an angle $\alpha$, as illustrated by the solid line  in the diagram. Now extend this segment  as a smooth closed curve of length $\ell + L$ with no further crossings (the extension is the dashed curve $\wt \gamma$).

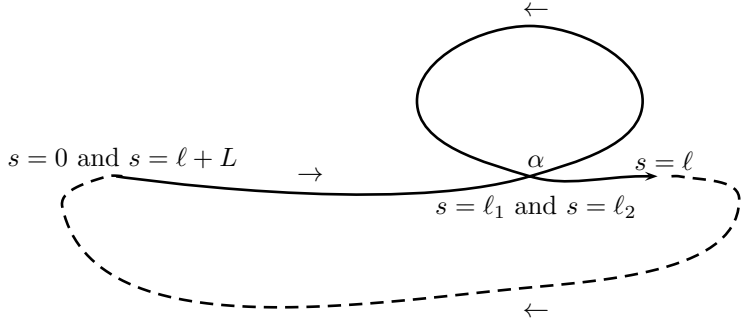
\pspicture*(4.5,-3.2)(17,1.5)
\pscurve[linewidth=1pt]{->}(7.0,-1)(7,-1)(12.5,-1)(14,0)(12.5,1)(11,0)(12.5,-1)(14,-1)(14.2,-1){->}     %
\pscurve[linewidth=1pt,linestyle=dashed](14.1,-1)(14.2,-1)(15,-1.3)(12,-2.5)(6,-1.4)(6.6,-1)(6.8,-1)
\rput(12.3,-.8){$\alpha$}%
\rput(12.3,-1.4){$s=\ell_1$ and $s=\ell_2$}%
\rput(6.8,-.75){$s=0$ and $s = \ell+L$}%
\rput(14,-.8){$s=\ell$}%
\rput(12.3,1.2){$\leftarrow$}%
\rput(12.3,-2.8){$\leftarrow$}%
\rput(9.3,-1){$\to$}%
\endpspicture

By the hypothesis of part (c),
$$\int_0^\ell \vt'(s)\,ds =0,
\text{
and by construction, }
\int_\ell^{\ell+L} \vt'(s) ds =-2\pi.
$$ So, from the hypothesis, the integral of $\vt'$ around the oriented loop $\gamma \cup \wt\gamma$ is $-2\pi$.
On the other hand, by the Hopf's Umlaufsatz for curvilinear polygons \cite[\S 13.2]{pres},
%
%
 $$\left|\int_{\ell_1}^{\ell_2} \vt'(s) ds\right| =\pi+\alpha = \left| \int_{\ell_2}^{\ell+L} \vt'(s) ds + \int_0^{\ell_1} \vt'(s) ds \right|.
 $$
 This is impossible since $\alpha \notin\{ 0,\,\pi\}$, because all crossings are simple. This contradiction completes the proof.

\begin{flushleft}
J. Howie\\
Department of Mathematics and Maxwell Institute for Mathematical Sciences\\
Heriot-Watt University\\
Edinburgh EH14 4AS
\end{flushleft}

\begin{flushleft}
J. F. Toland\\
Department of Mathematical Sciences\\
University of Bath\\
Bath   BA2 7AY\\
\end{flushleft}


\begin{thebibliography}{}

\bibitem{apos} T. M. Apostol and M.A. Mnatsakanian, Unwrapping curves from cylinders and cones. \emph{Amer. Math. Monthly.} \textbf{114} (2007), 388-416.



   \bibitem{chern} S. S. Chern,
Curves and surfaces in Euclidean space, in:
{\em Studies in Global Geometry and Analysis},  Studies in Mathematics Volume 4, Math. Asoc. Amer, 1967,
pp. 16--56.

\bibitem{pres}    A. Pressley, \emph{Elementary Differential Geometry.} Second Edition. Springer Undergraduate Mathematics Series. Springer, London, 2010.

\bibitem{aschur}
A.~Schur. {\"Uber die Schwarzsche Extremaleigenschaft des Kreises
unter den Kurven konstanter Kr\"ummung}. Mathematische Annalen
\textbf{83} (1921), 143-148. {\tt \footnotesize
http://www.digizeitschriften.de/main/dms/img/?PPN=PPN235181684}

    \bibitem{stein} H. Steinhaus, \emph{Mathematical Snapshots}. Dover, Mineola NY, 1999.

\bibitem{To7}
J.~F. Toland, Heavy hydroelastic travelling waves.  \emph{Proc. R. Soc.Lond.} A \textbf{463} (2007), 2371-2397
(DOI : 10.1098/rspa.2007.1883)


\end{thebibliography}
\end{document}